\documentclass[a4paper, 11pt]{amsart} 
\usepackage{mathptmx, amscd, eulervm, enumerate, colonequals, mathbbol, mathrsfs,amssymb}
\usepackage[dvipsnames]{xcolor}
\usepackage[pagebackref, colorlinks=true, citecolor=violet, linkcolor=blue, urlcolor=blue]{hyperref}

\usepackage[onehalfspacing]{setspace}
\usepackage{aliascnt}
\usepackage[inner=2.4cm,outer=2.4cm, bottom=3.2cm]{geometry}

\def\kk{\mathbb{k}}
\def\AA{\mathbb{A}}
\def\CC{\mathbb{C}}
\def\PP{\mathbb{P}}
\def\ZZ{\mathbb{Z}}
\def\QQ{\mathbb{Q}}

\def\calF{\mathcal{F}}
\def\calG{\mathcal{G}}
\def\calO{\mathcal{O}}
\def\calP{\mathcal{P}}
\def\calU{\mathcal{U}}

\def\ba{\mathbf{a}}

\def\bp{\mathbf{p}}

\def\bP{\mathbf{P}}

\def\frakA{\mathfrak{A}}
\def\frakM{\mathfrak{M}}

\def\frakn{\mathfrak{n}}
\def\frakp{\mathfrak{p}}

\def\scrE{\mathscr{E}}
\def\scrI{\mathscr{I}}
\def\scrL{\mathscr{L}}
\def\scrN{\mathscr{N}}
\def\Rees{\mathscr{R}}

\DeclareFontFamily{OT1}{pzc}{}
\DeclareFontShape{OT1}{pzc}{m}{it}{<-> s * [1.100] pzcmi7t}{}
\DeclareMathAlphabet{\mathchanc}{OT1}{pzc}{m}{it}
\DeclareMathOperator{\sHom}{\mathchanc{Hom}}
\DeclareMathOperator{\sProj}{\mathchanc{Proj}}
\DeclareMathOperator{\sSpec}{\mathchanc{Spec}}

\def\codim{\operatorname{codim}}
\def\rank{\operatorname{rank}}
\def\Ext{\operatorname{Ext}}
\def\gr{\operatorname{gr}}
\def\Frac{\operatorname{Frac}}
\def\HH{\operatorname{H}}
\def\Ass{\operatorname{Ass}}
\def\PGL{\operatorname{PGL}}
\def\Pic{\operatorname{Pic}}
\def\Proj{\operatorname{Proj}}
\def\Sing{\operatorname{Sing}}
\def\Spec{\operatorname{Spec}}
\def\maxSpec{\operatorname{MaxSpec}}
\def\Supp{\operatorname{Supp}}
\def\Sym{\operatorname{Sym}}
\def\Res{\operatorname{Res}}

\def\to{\longrightarrow}
\def\mapsto{\longmapsto}
\def\into{\lhook\joinrel\longrightarrow}

\newtheorem{theorem}{Theorem}[section]
\newtheorem{headthm}{Theorem}

\newaliascnt{corollary}{theorem}
\newtheorem{corollary}[corollary]{Corollary}
\aliascntresetthe{corollary}

\newaliascnt{lemma}{theorem}
\newtheorem{lemma}[lemma]{Lemma}
\aliascntresetthe{lemma}

\newaliascnt{proposition}{theorem}
\newtheorem{proposition}[proposition]{Proposition}
\aliascntresetthe{proposition}

\theoremstyle{definition}
\newaliascnt{definition}{theorem}
\newtheorem{definition}[definition]{Definition}
\aliascntresetthe{definition}

\newaliascnt{example}{theorem}

\aliascntresetthe{example}

\newaliascnt{remark}{theorem}
\newtheorem{remark}[remark]{Remark}
\aliascntresetthe{remark}

\newaliascnt{question}{theorem}
\newtheorem{question}[question]{Question}
\aliascntresetthe{question}

\newaliascnt{conjecture}{theorem}

\aliascntresetthe{conjecture}

\newaliascnt{hypothesis}{theorem}
\newtheorem{hypothesis}[hypothesis]{Hypothesis}
\aliascntresetthe{hypothesis}

\numberwithin{equation}{theorem}
\def\equationautorefname~#1\null{(#1)\null}

\title{Generic flatness of the cohomology of thickenings}

\author{Edoardo Ballico}
\address{Dipartimento di Matematica, Universit\`a di Trento, Via Sommarive, 14 - 38123 Povo (Trento), Italy}
\email{edoardo.ballico@unitn.it}
\urladdr{https://orcid.org/0000-0002-1432-7413}

\author{Yairon Cid-Ruiz}
\address{Department of Mathematics, North Carolina State University, SAS Hall 4214, Raleigh, NC~27695,~USA}
\email{ycidrui@ncsu.edu}

\author{Anurag K. Singh}
\address{Department of Mathematics, University of Utah, 155 South 1400 East, Salt Lake City, UT~84112, USA}
\email{singh@math.utah.edu}

\date{\today}

\thanks{Y.C.-R. was supported by NSF grant DMS~2502321, and A.K.S. by NSF grants DMS~2101671 and DMS~2349623; E.B. is a member of GNSAGA of INdAM (Italy)}

\begin{document}

\begin{abstract}
We prove a generic flatness result for the cohomology of thickenings of a projective scheme that is smooth over a Noetherian domain containing a field of characteristic zero. Our study is motivated, in part, by a classical question in algebraic geometry: Given a set of $m$ distinct points in projective space over a field, and $t$ a positive integer, determine the least degree of a hypersurface that passes through each point with multiplicity at least $t$. Related to this, it remains unresolved whether there exists a dense open set of $m$-tuples of points for which this least degree is constant for each $t\ge 1$. Investigating this connection in the case of nine points in projective plane, we construct a local cohomology module that is not generically free; moreover, we show that it has infinitely many associated prime ideals.
\end{abstract}

\maketitle

\section{Introduction}

Let $X \subset \PP_\kk^n$ be a projective scheme over a field $\kk$. For each integer $t \ge 1$, the \emph{$t$-th thickening} of $X$ is the closed subscheme $X_t \colonequals V\left(\scrI_X^t\right)$ in $\PP_\kk^n$, where $\scrI_X \subset \calO_{\PP_\kk^n}$ is the ideal sheaf of $X$. The study of the asymptotic behavior of the cohomology of these thickenings has a long history in algebraic geometry and commutative algebra, with connections to deformation theory, Rees algebras, and local cohomology. More precisely, given a coherent sheaf $\calF$ on $\PP_\kk^n$ that is flat along $X$, one is interested in asymptotically studying the induced projective system
\[
\CD
@>>> \HH^i\big(X_{t+1},\, \calO_{X_{t+1}}\otimes_{\calO_{\PP^n_\kk}}\calF\big)
@>>> \HH^i\big(X_t,\, \calO_{X_t}\otimes_{\calO_{\PP^n_\kk}}\calF\big)
@>>> \HH^i\big(X_{t-1},\, \calO_{X_{t-1}}\otimes_{\calO_{\PP^n_\kk}}\calF\big) 
@>>>
\endCD
\]
for each cohomological index $i \ge 0$. We summarize some work in this direction:

(a) A foundational result of Hartshorne \cite{HARTSHORNE_AMPLE} shows that if $X \subset \PP_\kk^n$ is smooth, where $\kk$ has characteristic zero, then the cohomology of the symmetric powers of the conormal bundle exhibits strong vanishing properties; as a consequence, for each $i<\dim X$, the natural map
\begin{equation}
\label{equation_map_thickenings}
\CD
\HH^i\big(X_{t+1},\, \calO_{X_{t+1}}\otimes_{\calO_{\PP^n_\kk}}\calF\big)
@>>> \HH^i\big(X_t,\, \calO_{X_t}\otimes_{\calO_{\PP^n_\kk}}\calF\big)
\endCD
\end{equation}
is an isomorphism for all $t\gg0$; see \cite[Theorem~8.1~(b)]{HARTSHORNE_AMPLE} and \cite[Theorem~2.8]{BBLSZ2}.

(b) This is extended in \cite[Theorem~1.1]{BBLSZ2} as follows: if $X\subset\PP_\kk^n$ is locally a complete intersection, where $\kk$ is a field of characteristic zero, then, for each $i<\codim(\Sing X)$, the map~\autoref{equation_map_thickenings} is an isomorphism for all~$t\gg0$. In particular, if $X$ has at most isolated singular points, then~\autoref{equation_map_thickenings} is an isomorphism for all~$i<\dim X$ and $t\gg0$.

(c) A version of the Kodaira vanishing theorem for thickenings of locally complete intersection subvarieties of $\PP_\kk^n$, for $\kk$ of characteristic zero, may be found in \cite[Theorem~1.4]{BBLSZ2}; an asymptotic version of the same holds in positive prime characteristic, \cite[Theorem~1.2]{BBLSZ3}.

The present paper investigates a relative version of these phenomena over a Noetherian base ring; specifically, we study the generic flatness of the cohomology of thickenings of a closed subscheme in the following context: Let $A$ be a Noetherian domain, and $X\subset\PP^n_A$ a closed subscheme. Consider the $t$-th thickening $X_t\colonequals V(\scrI_X^t)$ of $X$ in $\PP_A^n$, where $\scrI_X\subset\calO_{\PP_A^n}$ is the ideal sheaf of $X$. For $\calF$ a coherent sheaf on~$\PP_A^n$ that is flat along $X$, a natural question then arises regarding the behavior of the cohomology groups 
\[
\HH^i\big(X_t,\, \calO_{X_t}\otimes_{\calO_{\PP_A^n}}\calF\big) 
\]
as $t$ varies: generic flatness guarantees that for a \emph{fixed} $t$, the cohomology is flat over a dense open subset of $\Spec A$; however, it is not clear whether there is a single dense open subset that suffices for \emph{all} $t \ge 1$. Our first main result provides an affirmative answer to this when $X$ is smooth over $\Spec A$:

\begin{headthm}
\label{thm_A}
Let $A$ be a Noetherian domain containing a field of characteristic zero. Let $X\subset\PP^n_A$ be a closed subscheme that is smooth over $\Spec A$, and $X_t$ be its $t$-th thickening, as defined above. Fix an integer~$j$, and set $\calF\colonequals\calO_{\PP^n_A}(j)$. Then there exists a nonzero element $a\in A$ such that 
\[
\HH^i\big(X_t,\, \calO_{X_t}\otimes_{\calO_{\PP^n_A}}\calF\big)\otimes_A A_a
\]
is a flat $A_a$-module for each integer $i\ge 0$, for all thickenings $t\ge 1$.
\end{headthm}

The proof of this relies on understanding the asymptotic behavior of the cohomology of powers of the ideal sheaf. Towards this, we use the notion of relatively ample vector bundles; while the theory of ample vector bundles on schemes over a field is well-established following Hartshorne \cite{HARTSHORNE_AMPLE}, we require a relative version over an arbitrary Noetherian base scheme $S$. In \autoref{sect_ample_vect}, we develop the necessary theory of $f$-ample vector bundles and establish a fiberwise criterion for ampleness, where $f\colon X\subset\PP_S^n\to S$ is a projective morphism. Our primary example of an $f$-ample vector bundle is the normal bundle $\scrN_X$ associated with a closed immersion $X\into\PP^n_S$, assuming that $X$ is smooth over the base $S$. 

A key technical tool in our approach is a uniform vanishing result, \autoref{thm_uniform_vanish}, for the cohomology of the symmetric powers of the conormal bundle, that extends classical results of Hartshorne. In \autoref{sect_flatness}, we combine this vanishing result with a generic freeness result for local cohomology modules established in \cite{CRS}. More precisely, we exploit the following: 

(a) In \cite[Theorem~B]{CRS}, it was shown that the local cohomology modules $\HH_I^i(R)$ are generically free over the base ring $A$, where $I$ is an ideal of $R\colonequals A[x_0,\dots,x_n]$.

(b) Using a graded local duality argument, we obtain a stabilization result for the graded components of the modules $\Ext_R^i(R/I^t,\, R)$, associated with the thickenings of $V(I)$ in $\PP_A^n$. 

(c) We deduce the flatness of the cohomology of thickenings by exploiting the natural identification
\[
\lim\limits_{\substack{\to\\ t}} \Ext_R^i\big(R/I^t,\, R\big)\ =\ \HH_I^i(R),
\]
and the aforementioned stabilization result for Ext modules; see \autoref{prop_inject} and \autoref{prop_valuative}.

The generic freeness of local cohomology modules of certain Rees algebras connects to a classical question regarding points in projective space:

\begin{question}
\label{question_nagata}
Given a set $X$ of $m$ distinct points in $\PP_\kk^n$, what is the minimum degree $\alpha_t(X)$ of a hypersurface in~$\PP_\kk^n$ that passes through each point with multiplicity at least $t$?
\end{question}

The question gained significance following Nagata's work on Hilbert's 14-th problem~\cite{NAGATA_CONJ}; towards constructing his counterexample, Nagata proved that if $m\ge 16$ is a perfect square, then
\[
\alpha_t(X)\ >\ t\sqrt{m}.
\]
Substantial effort has since been put forth towards \autoref{question_nagata}, though it remains elusive and difficult; cf.~\cite{CHUDNOVSKY_CONJ, DEMAILLY, ESNAULT_VIEHWEG, BOCCI_HARBOURNE, DUMNICKI_P3, HARBOURNE_HUNEKE, GHM2013, Dumnicki2015, DUMNICKI_TUTAJ, FOULI_MANTERO_XIE, MSS2018, BISUI_CHUDNOVSKY}. By a celebrated theorem of Zariski and Nagata, see \cite{ZARISKI, NAGATA_LOCAL_RINGS}, \cite[Theorem~3.14]{EISEN_COMM}, when $\kk$ is perfect, one has
\[
\alpha_t(X)\ =\ \alpha\big(I_X^{(t)}\big)\quad \text{ for all } t\ge 1,
\]
where $I_X$ is the defining ideal of $X$, the ideal $I_X^{(t)}$ is its $t$-th symbolic power, and for $I$ a homogeneous ideal, $\alpha(I)$ denotes the least degree of a homogeneous element of $I$.

Given an $m$-tuple $\bp\colonequals(\bp_1,\dots,\bp_m)$ of distinct points from $\PP_\kk^n$, set $X_\bp\colonequals\{\bp_1,\dots,\bp_m\}$ in $\PP_\kk^n$. To the best of our knowledge the following is completely open:

\begin{question}
\label{question_constant}
Does there exist a dense open subset $\calU$ in $\left(\PP_\kk^n\right)^m$, such that the function 
\[
\bp\in\calU\ \mapsto\ \alpha_t\big(X_\bp\big)
\]
is constant for all $t\ge 1$?
\end{question}

The question above asks whether the invariants $\alpha_t(X)$ are all constant for a \emph{general set} of $m$ points in~$\PP_\kk^n$. Again, a generic flatness argument yields such a dense open subset for a \emph{fixed} $t$, but the challenge lies in finding a single dense open subset that works for \emph{all} $t \ge 1$. In \autoref{sect_points}, we note that an affirmative answer to \autoref{question_constant} would follow from the generic freeness of local cohomology modules of a Rees algebra; we prove this generic freeness for the case of up to $n+2$ points in projective $n$-space over a field, see \autoref{thm_points} (where the answer to \autoref{question_constant} is known to be affirmative). However, in \autoref{nine_points}, we show that the hoped-for generic freeness fails in the case of nine points in $\PP^2$. This, in turn, yields a local cohomology module that has infinitely many associated primes --- a construction of independent interest (see \autoref{thm_9points}):

\begin{headthm}
Let $\ba\colonequals\left\{a_{ij} \mid 1\le i\le 9,\ 0\le j \le 2\right\}$ be a set of indeterminates over $\CC$, and set $A\colonequals\CC[\ba]$.
Let $R$ be the polynomial ring $A[x_0,x_1,x_2]$, and for $1\le i\le 9$, consider the ideals
\[
\frakA_i \colonequals I_2 \begin{pmatrix}
x_0 & x_1 & x_2 \\
a_{i0} & a_{i1} & a_{i2}
\end{pmatrix}.
\]
Set $I\colonequals\frakA_1\cap\dots\cap\frakA_9 \subset R$ and consider the corresponding Rees algebra $\Rees(I) = \bigoplus_{t\ge 0} I^tT^t$. 
Then the following statements hold: 
\begin{enumerate}[\quad\rm(a)]
\item There is no nonzero element $a \in A$ such that $\HH_{(x_0,x_1,x_2)}^2\left(\Rees(I)\right) \otimes_A A_a$ is a flat $A_a$-module.
\item Each of the following sets of associated primes is infinite:
\[
\Ass_A\left(\HH_{(x_0,x_1,x_2)}^2\left(\Rees(I)\right)\right), \quad
\Ass_R\left(\HH_{(x_0,x_1,x_2)}^2\left(\Rees(I)\right)\right),
\text{\quad and \quad}
\Ass_{\Rees(I)}\left(\HH_{(x_0,x_1,x_2)}^2\left(\Rees(I)\right)\right).
\]
\end{enumerate}
\end{headthm}

This result highlights why the symbolic powers of an ideal of points have historically been, and remain, notoriously difficult to analyze. Indeed, it shows that the symbolic powers of an ideal of points will inevitably have a non-uniform behavior (see \autoref{cor_symb_pow}).

\smallskip

While the question of the finiteness of the Bass numbers of local cohomology modules was raised by Grothendieck~\cite[Expos\'e XIII, Conjecture 1.1]{SGA2} and answered in the negative by Hartshorne~\cite{HARTSHORNE_COFINITENESS}, the question of the finiteness of the set of associated prime ideals is due to Huneke~\cite[Problem~4]{HUNEKE}. Affirmative answers include the case of regular rings of prime characteristic~\cite{HUNEKE_SHARP}, regular local and affine rings of characteristic zero~\cite{LYUBEZNIK93}, unramified regular local rings of mixed characteristic~\cite{LYUBEZNIK00}, smooth $\ZZ$-algebras~\cite{BBLSZ1}, and rings of positive prime characteristic with finite Frobenius representation type \cite{TAKAGI_TAKAHASHI, HOCHSTER_NUNEZ, DAO_QUY}; other positive results over rings of low dimension may be found in~\cite{MARLEY}. The answer, however, is negative in general; the counterexamples in \cite{SINGH, JEFFRIES, JEFFRIES_SINGH} come from integer torsion in the local cohomology of finitely generated algebras over the ring of integers, while those in \cite{KATZMAN, SINGH_SWANSON} are for affine algebras over a field. The example constructed in \autoref{nine_points} is arguably more geometric in flavor: it arises from torsion points in elliptic curves over the complex numbers.

\section{Ample vector bundles relative to a base scheme}
\label{sect_ample_vect}

The notion of \emph{ample vector bundles} on a scheme of finite type over a field was introduced by Hartshorne \cite{HARTSHORNE_AMPLE}; his main motivation was precisely to study the cohomology of thickenings as in~\autoref{equation_map_thickenings}. For our purposes, we require a relative version of this notion over an arbitrary Noetherian base scheme. While this extension is possibly well-known to experts, we were unable to locate a reference that treats it in a form suitable for our needs. Accordingly, in this short section, we develop the necessary relative theory.

Let $S$ be a Noetherian scheme and $f\colon X\subset\PP_S^n \to S$ be a projective morphism. Let $\scrE$ be a \emph{vector bundle} on $X$, i.e., a locally free sheaf of finite rank. Consider the projectivization 
\[
\CD
\bP(\scrE)\colonequals\sProj_X\left(\Sym_{\calO_X}(\scrE)\right) @>\pi>>X,
\endCD
\]
where we follow Grothendieck’s convention regarding projective bundles. Let 
\[
g = f\circ\pi\colon \bP(\scrE)\ \to \ S
\]
denote the composition of $f$ and $\pi$. By extending the notion of a relatively ample line bundle, see, for example, \cite[\href{https://stacks.math.columbia.edu/tag/01VG}{Tag 01VG}]{stacks-project}, we obtain the following definition:

\begin{definition}
A vector bundle $\scrE$ on $X$ is \emph{$f$-ample} if the canonical line bundle $\calO_{\bP(\scrE)}(1)$ is $g$-ample.
\end{definition}

Let $\kappa(s)$ denote the residue field at a point $s\in S$. Consider the fiber
\[
X_s\ \colonequals\ X \times_S\Spec(\kappa(s))\ \subset\ \PP_{\kappa(s)}^n
\]
and the restriction vector bundle $\scrE_s \colonequals \scrE \otimes_{\calO_S} \kappa(s)$ on $X_s$. Using the fiberwise criterion for the ampleness of line bundles, \cite[Th\'eor\`eme~4.7.1]{EGAIII_1} and \cite[Corollaire~9.6.4]{EGAIV_3}, we derive a corresponding fiberwise criterion for the ampleness of vector bundles:

\begin{theorem}
\label{thm_fiberwise_crit}
The vector bundle $\scrE$ on $X$ is $f$-ample if and only if the restriction vector bundle $\scrE_s$ on $X_s$ is ample for each $s\in S$.
\end{theorem}

\begin{proof}
The aforementioned fiberwise criterion implies that $\calO_{\bP(\scrE)}(1)$ is $g$-ample if and only if
\[
\calO_{\bP(\scrE_s)}(1)\ \cong\ \calO_{\bP(\scrE)}(1)\otimes_{\calO_S}\kappa(s)
\]
is ample for each $s\in S$. 
\end{proof}

Next, assume that $f\colon X\subset\PP_S^n\to S$ is a smooth morphism. Our object of interest is the normal bundle to $X$ in $\PP_S^n$, namely
\[
\scrN_X \colonequals \sHom_{\calO_X}\big(\scrI_X/\scrI_X^2,\, \calO_X\big).
\]

\begin{lemma}
\label{lem_ample_vect}
If $f\colon X \subset\PP_S^n\to S$ is a smooth morphism, then the normal bundle $\scrN_X$ is $f$-ample.
\end{lemma}

\begin{proof}
For each $s\in S$, \cite[Proposition~III.2.1]{HARTSHORNE_AMPLE_NOTES} implies that the normal bundle to $X_s$ in $\PP_{\kappa(s)}^n$, i.e., $\scrN_X\otimes_{\calO_S}\kappa(s)$, is an ample vector bundle on $X_s$. Therefore $\scrN_X$ is $f$-ample by \autoref{thm_fiberwise_crit}.
\end{proof}

Our main interest here lies in the following extension of a well-known vanishing result of Hartshorne, \cite[Theorem~8.1]{HARTSHORNE_AMPLE}; see also \cite[Theorem~2.8]{BBLSZ2}. A key point in the vanishing result below is the existence of a stabilization index that is uniform across all fibers.

\begin{theorem}
\label{thm_uniform_vanish}
Let $S$ be a Noetherian scheme over the field of rational numbers, and $f\colon X\subset\PP_S^n\to S$ be a smooth morphism of relative dimension $d\ge 0$. Let $\calF$ be a vector bundle on $X$, and set $\calF_s\colonequals \calF\otimes_{\calO_S}\kappa(s)$ for $s\in S$. Then there exists an integer $t_0>0$ such that 
\[
\HH^i\left(X_s,\, \Sym^t\big(\scrI_{X_s}/\scrI_{X_s}^2\big)\otimes_{\calO_{X_s}}\calF_s\right)\ =\ 0
\]
for all $s\in S$, $t\ge t_0$, and $i\neq d$.
\end{theorem}

\begin{proof}
Let $s \in S$. By Serre duality, e.g., \cite[\S III.7]{HARTSHORNE}, the cohomology group displayed above is dual to the cohomology group 
\[
\HH^{d-i}\Big(X_s,\, \left(\Sym^t(\scrI_{X_s}/\scrI_{X_s}^2)\right)^\vee\otimes_{\calO_{X_s}}\calF_s^\vee \otimes_{\calO_{X_s}}\omega_{X_s}\Big),
\]
where $(-)^\vee$ denotes $\sHom_{\calO_{X_s}}(-,\, \calO_{X_s})$. The characteristic zero assumption yields a natural isomorphism 
\[
\sHom_{\calO_X}\left(\Sym^t(\scrI_X/\scrI_X^2),\, \calO_X\right)\ \cong\ \Sym^t(\scrN_X),
\]
see, for example, \cite[Proposition~A2.7]{EISEN_COMM}. Consider the vector bundle
\[
\calG \colonequals \calF^\vee \otimes_{\calO_X} \omega_{X/S}
\]
on $X$. After the above considerations, it suffices to find an integer $t_0 > 0$ such that 
\[
\HH^j\big(X_s,\, \Sym^t(\scrN_X)\otimes_{\calO_X}\calG\otimes_{\calO_S}\kappa(s)\big)\ =\ 0
\]
for all $s\in S$, $t\ge t_0$, and $j\neq 0$. We now proceed to prove this claim. 

Let $\scrN=\scrN_X$, $\pi\colon \bP(\scrN)\to X$, and $g=f\circ\pi \colon \bP(\scrN)\to S$. By \autoref{lem_ample_vect}, $\scrN$ is an $f$-ample vector bundle on $X$, so $\calO_{\bP(\scrN)}(1)$ is a $g$-ample line bundle on $\bP(\scrN)$. Consider the Grothendieck spectral sequence 
\[
E_2^{p,q}\ =\ R^pf_*\left(R^q\pi_*\big(\calO_{\bP(\scrN)}(t) \otimes_{\calO_{\bP(\scrN)}} \pi^*(\calG)\big)\right)
\ \Longrightarrow\
R^{p+q}g_*\left(\calO_{\bP(\scrN)}(t) \otimes_{\calO_{\bP(\scrN)}} \pi^*(\calG)\right).
\]
We have the isomorphism $\pi_*\big(\calO_{\bP(\scrN)}(t) \otimes_{\calO_{\bP(\scrN)}} \pi^*(\calG)\big)\ \cong\ \Sym^t(\scrN)\otimes_{\calO_X}\calG$ and the vanishing 
\[
R^q\pi_*\left(\calO_{\bP(\scrN)}(t) \otimes_{\calO_{\bP(\scrN)}} \pi^*(\calG)\right)\ =\ 0
\quad \text{ for all } q \ge 1;
\]
see \cite[Lemma~3.1]{HARTSHORNE_AMPLE}, \cite[Exercise~III.8.4]{HARTSHORNE}. Hence, the above spectral sequence degenerates and yields the isomorphism
\[
R^jf_*\left(\Sym^t(\scrN) \otimes_{\calO_X} \calG\right)
\ \cong\
R^jg_*\left(\calO_{\bP(\scrN)}(t) \otimes_{\calO_{\bP(\scrN)}} \pi^*(\calG)\right) 
\quad \text{ for all } j \ge 0.
\]
Since $\calO_{\bP(\scrN)}(1)$ is $g$-ample, there is an integer $t_0 > 0$ such that 
\[
R^jf_*\left(\Sym^t(\scrN)\otimes_{\calO_X}\calG\right)\ \cong\
R^jg_*\left(\calO_{\bP(\scrN)}(t)\otimes_{\calO_{\bP(\scrN)}} \pi^*(\calG)\right)\ =\ 0
\]
for all $j \ge 1$ and $t \ge t_0$, see \cite[\href{https://stacks.math.columbia.edu/tag/02O1}{Tag 02O1}]{stacks-project}, \cite[Theorem~III.8.8]{HARTSHORNE}.

Fix $t \ge t_0$. Consider the natural base change map 
\[
\CD
\varphi^j(s)\colon R^jf_*\left(\Sym^t(\scrN)\otimes_{\calO_X}\calG\right)\otimes_{\calO_S} \kappa(s)
@>>>
\HH^j\left(X_s,\, \Sym^t(\scrN_s)\otimes_{\calO_{X_s}}\calG_s\right).
\endCD
\]
The cohomology and base change theorem, \cite[Theorem~III.12.11]{HARTSHORNE}, \cite[\S7.7]{EGAIII_2}, \cite[\S5]{MUMFORD_ABELIAN}, tells us that: (a) $\varphi^j(s)$ is an isomorphism if it is surjective; and (b) if $\varphi^j(s)$ is surjective, then $\varphi^{j-1}(s)$ is surjective if and only if $R^jf_*\left(\Sym^t(\scrN) \otimes_{\calO_X}\calG\right)$ is locally free in a neighborhood of $s$.

Finally, by descending induction on $j\le d$, since $\varphi^d(s)$ is surjective, we obtain the vanishing 
\[
\HH^j\left(X_s,\, \Sym^t(\scrN_s)\otimes_{\calO_{X_s}}\calG_s\right)\ =\ 0
\]
for all $j \ge 1$, see \cite[\S5,~Corollary 4]{MUMFORD_ABELIAN}. This completes the proof of the theorem.
\end{proof}

\section{Generic flatness in the smooth setting}
\label{sect_flatness}

In this section, we complete the proof of \autoref{thm_A}, i.e., under a smoothness hypothesis, we prove that there exists a flatness locus for the cohomology of all thickenings. The main tools are the vanishing result, \autoref{thm_uniform_vanish}, and the generic freeness result for local cohomology modules stated in \autoref{thm_CRS}.

We fix the notation that we will use throughout this section. Let $A$ be a Noetherian domain containing a field of characteristic zero. Let $R\colonequals A[x_0,\dots,x_n]$ be a standard graded polynomial ring over $A$, i.e., $\deg(x_i)=1$, and $\deg(a)=0$ for $a \in A$. Set $\frakM \colonequals (x_0,\dots,x_n)$ and $\PP_A^n \colonequals \Proj(R)$.

Let $Q\colonequals \Frac A$ be the field of fractions of $A$. For $\frakp\in\Spec A$, let $\kappa(\frakp) \colonequals A_\frakp/\frakp A_\frakp$ be the residue field at $\frakp$, and set
\[
R(\frakp)\colonequals R\otimes_A\kappa(\frakp)\ \cong\ \kappa(\frakp)[x_0,\dots,x_n].
\]
For $M$ a finitely generated graded $R$-module, set $M(\frakp)\colonequals M\otimes_A\kappa(\frakp)$.

Let $X\subset\PP_A^n$ be a closed subscheme with ideal sheaf $\scrI_X\subset\calO_{\PP_A^n}$. For $\frakp\in\Spec A$, consider the fiber 
\[
X_\frakp\ \colonequals\ X\times_{\Spec A}\Spec(\kappa(\frakp))\ \subset\ \PP_{\kappa(\frakp)}^n.
\]
For $t \ge 1$, the \emph{$t$-th thickening} of $X$ in $\PP_A^n$ is the closed subscheme of $\PP_A^n$ determined by the $t$-th power of the ideal sheaf $\scrI_X$ of $X$, namely
\[
X_t\colonequals V\left(\scrI_X^t\right)\ \subset\ \PP_A^n.
\]
For a coherent sheaf $\calF$ on $X$, set $\calF_\frakp \colonequals \calF \otimes_A \kappa(\frakp)$.

The following proposition yields a version of the graded local duality theorem relative to our Noetherian base ring $A$. Results of this type have also appeared in \cite{HOCHSTER_ROBERTS, KSMITH, GEN_FREENESS_LOC_COHOM, FIBER_FULL}.

\begin{proposition}
\label{prop_fib_full_mod}
Let $R= A[x_0,\dots,x_n]$ be a standard graded polynomial ring over $A$, and $M$ a finitely generated graded $R$-module that is flat over $A$. Fix an integer $j$. Then the following are equivalent:

\begin{enumerate}[\quad\rm(a)]
\item ${\HH_\frakM^i(M)}_j$ is a flat $A$-module for each $i\ge 0$.

\item ${\Ext_R^i(M,\, R)}_{-j-n-1}$ is a flat $A$-module for each $i\ge 0$.

\item The function $\Spec A\to\ZZ$ where $\frakp\mapsto\rank_{\kappa(\frakp)} {\HH_\frakM^i\left(M(\frakp)\right)}_j$ is constant for each $i\ge 0$.

\item The function $\Spec A\to\ZZ$ where
\[
\frakp\mapsto\rank_{\kappa(\frakp)} {\Ext_{R(\frakp)}^i\left(M(\frakp),\, R(\frakp)\right)}_{-j-n-1}
\]
is constant for each~$i \ge 0$.
\end{enumerate}
Moreover, when these equivalent conditions are satisfied, we have the base change isomorphisms
\[
\CD
{\HH_\frakM^i\left(M\right)}_j \otimes_A B @>\cong>> {\HH_\frakM^i\left(M\otimes_A B\right)}_j
\quad\text{ and }\quad
{\Ext_R^i\left(M,\, R\right)}_j \otimes_A B @>\cong>> {\Ext_{R_B}^i\left(M\otimes_A B,\, R_B\right)}_j
\endCD
\]
for each $i\ge 0$, where $B$ is an arbitrary $A$-algebra, and $R_B \colonequals R \otimes_A B$.
\end{proposition}

\begin{proof}
The equivalence (a) $\iff$ (b) follows from \cite[Theorem~A]{FIBER_FULL}; one can restrict this result to the $j$-th graded component, see also the proof of \cite[Proposition~2.11]{FIBER_FULL}. The base change isomorphisms may also be obtained from \cite[Theorem~A]{FIBER_FULL}.

The equivalence (a) $\iff$ (c) follows from classical results on cohomological functors, see for example \cite[Th\'eor\`eme~7.6.9]{EGAIII_2} or \cite[Corollary~12.9]{HARTSHORNE}. The equivalence (c) $\iff$ (d) may be obtained by applying the graded local duality theorem \cite[Theorem~3.6.19]{BRUNS_HERZOG} to the module $M(\frakp)$.
\end{proof}

For our purposes, one of the principal applications of \autoref{prop_fib_full_mod} is the following corollary; this enables us to apply the vanishing proven in \autoref{thm_uniform_vanish}.

\begin{corollary}
\label{cor_vanish_ext}
Let $I \subset R$ be a homogeneous ideal defining a closed subscheme $X \colonequals V(I) \subset \PP_A^n$. Let $M$ be a finitely generated $R/I$-module, and $\calF \colonequals \widetilde{M}$ be the associated coherent sheaf on $X$. Fix a nonnegative integer $d \ge 0$. Assume that:

\begin{enumerate}[\quad\rm(a)]
\item $M$ is $A$-flat,
\item $M_0=0$,
\item $\HH^i\big(X_\frakp,\, \calF_\frakp\big) = 0$ for all $i\neq d$ and $\frakp\in\Spec A$.
\end{enumerate}
Then ${\Ext_R^{n-d}(M,\, R)}_{-n-1}$ is $A$-flat, and ${\Ext_R^{n-i}(M,\, R)}_{-n-1}=0$ for all $i\neq d$. Moreover, we have 
\[
{\Ext_R^i(M,\, R)}_{-n-1} \otimes_A B\ \cong\ {\Ext_{R_B}^i(M_B,\, R_B)}_{-n-1}
\] 
for each $i\ge 0$, where $B$ is an arbitrary $A$-algebra, $R_B\colonequals R\otimes_A B$, and $M_B\colonequals M\otimes_A B$.
\end{corollary}

\begin{proof}
We have an exact sequence 
\[
\CD
0 @>>> \HH_\frakM^0\left(M(\frakp)\right) @>>> M(\frakp) @>>> \bigoplus_{j\in\ZZ}\, \HH^0\big(X_\frakp,\, \calF_\frakp(j)\big) @>>> \HH_\frakM^1\left(M(\frakp)\right) @>>> 0
\endCD
\]
and an isomorphism
\[
\HH_\frakM^{i+1}\left(M(\frakp)\right) \,\cong\, \bigoplus_{j\in\ZZ}\HH^i\big(X_\frakp,\, \calF_\frakp(j)\big)\quad \text{ for all } i\ge 1.
\]
From the assumptions, we obtain 
\[
{\HH_\frakM^{i+1}\left(M\left(\frakp\right)\right)}_0\ =\ 0 \quad \text{ for all } i\neq d.
\] 
Moreover, since the function 
\[
\frakp\in\Spec A\ \mapsto\ \chi\left(\calF_\frakp\right)\ =\ \sum_{i\ge 0}(-1)^i\rank_{\kappa(\frakp)} \HH^i(X_\frakp,\, \calF_\frakp)
\]
is constant, \cite[\S5]{MUMFORD_ABELIAN}, we conclude that $\frakp\in\Spec A \mapsto \rank_{\kappa(\frakp)} {\HH_\frakM^{d+1}\left(M\left(\frakp\right)\right)}_0$ is constant as well. The result now follows from \autoref{prop_fib_full_mod}.
\end{proof}

We next recall a result on the generic freeness of local cohomology modules from \cite{CRS}:

\begin{theorem}[{\cite[Theorem~B]{CRS}}]
\label{thm_CRS}
Let $A$ be a Noetherian domain containing a field of characteristic zero, and $I$ be an ideal of $R = A[x_0,\dots,x_n]$. Then there is a nonzero element $a \in A$ such that
\[
\HH_I^i(R) \otimes_A A_a
\]
is $A_a$-free for each $i \ge 0$.
\end{theorem}

\begin{remark}
When $\HH_I^i(R)\otimes_A A_a$ is $A_a$-flat for each $i \ge 0$, one has base change isomorphisms
\[
\CD
\HH_I^i(R)\otimes_A B @>\cong>> \HH_I^i(R\otimes_A B)
\endCD
\]
for each $i \ge 0$, where $B$ is an arbitrary $A_a$-algebra.
\end{remark}

The next proposition provides a crucial stabilization argument; it allows us to exploit the generic flatness of local cohomology, \autoref{thm_CRS}, to study Ext modules:

\begin{proposition}
\label{prop_inject}
Let $I \subset R$ be a homogeneous ideal and $X=V(I)\subset\PP_A^n$ be the closed subscheme defined by $I$. Fix an integer $j\in\ZZ$. Assume that: 

\begin{enumerate}[\quad\rm(a)]
\item The natural projection $X\subset\PP_A^n\to\Spec A$ is a smooth morphism of relative dimension $d\ge 0$,
\item $I^t/I^{t+1}$ is $A$-flat for each $t\ge 0$,
\item $\HH_I^i(R)$ is $A$-flat for each $i\ge 0$.
\end{enumerate}
Then there is an integer $t_0>0$, such that the natural map 
\[
\CD
\psi_{B,t}^i\colon {\Ext_{R_B}^i\left(R_B/I^tR_B,\, R_B\right)}_{-j-n-1} @>>> {\Ext_{R_B}^i\left(R_B/I^{t+1}R_B,\, R_B\right)}_{-j-n-1}
\endCD
\]
is injective for each $t\ge t_0$ and $i\ge 0$, where $B$ is an arbitrary $A$-algebra, and $R_B\colonequals R\otimes_A B$.
\end{proposition}

\begin{proof}
Set $k\colonequals -j-n-1$, and let $B$ be an $A$-algebra. Taking $\calF = \calO_X(j)$ in \autoref{thm_uniform_vanish}, there exists an integer $t_0>0$ such that 
\[
\HH^i\Big(X_\frakp,\, \Sym^t\big(\scrI_{X_\frakp}/\scrI_{X_\frakp}^2\big)(j)\Big)\ =\ 0
\]
for all $\frakp\in\Spec A$, $t\ge t_0$, and $i\neq d$. After possibly increasing $t_0$, we may assume that $\left[I^t/I^{t+1}\right]_j=0$ for all $t\ge t_0$. Applying \autoref{cor_vanish_ext} to $M=\left(I^t/I^{t+1}\right)(j)$ and $\widetilde{M}\cong\Sym^t(\scrI_X/\scrI_{X}^2)(j)$ yields
\[
{\Ext_{R_B}^{n-i}\left(I^tR_B/I^{t+1}R_B,\, R_B\right)}_k\ =\ 0
\]
for all $t\ge t_0$ and $i\neq d$.

For $t\ge t_0$, consider the exact sequence 
\[
\CD
0 @>>> I^tR_B/I^{t+1}R_B @>>> R_B/I^{t+1}R_B @>>> R_B/I^tR_B @>>> 0.
\endCD
\]
Setting $c \colonequals n-d$, the associated exact sequence of Ext modules gives: 
\begin{alignat*}3
\psi_{B,t}^c \colon & {\Ext_{R_B}^{c}(R_B/I^tR_B,\, R_B)}_k & \to\ & {\Ext_{R_B}^c(R_B/I^{t+1}R_B,\, R_B)}_k & \quad & \text{is injective;}\\
\psi_{B,t}^{c+1} \colon & {\Ext_{R_B}^{c+1}(R_B/I^tR_B,\, R_B)}_k & \to\ & {\Ext_{R_B}^{c+1}(R_B/I^{t+1}R_B,\, R_B)}_k & \quad & \text{is surjective; and}\\
\psi_{B,t}^i \colon & {\Ext_{R_B}^i(R_B/I^tR_B,\, R_B)}_k & \to\ & {\Ext_{R_B}^i(R_B/I^{t+1}R_B,\, R_B)}_k & \quad & \text{is an isomorphism for $i \neq c, c+1$}.
\end{alignat*}
It only remains to deal with the case $i=c+1$. For each $t\ge t_0$, we have an exact sequence 
\[
\CD
0 @>>> K_t @>>> {\Ext_R^{c+1}(R/I^{t_0},\, R)}_k @>>> {\Ext_R^{c+1}(R/I^t,\, R)}_k @>>> 0,
\endCD
\]
with $K_t$ denoting the kernel. By the Noetherian hypothesis, the ascending sequence 
\[
K_{t_0}\, \subseteq\, K_{t_0+1}\, \subseteq\, K_{t_0+2}\, \subseteq\, \dots
\]
inside ${\Ext_R^{c+1}(R/I^{t_0},\, R)}_k$ must be eventually constant. Thus, after possibly increasing $t_0$, we may assume that $\psi_{A,t}^{c+1}$ is an isomorphism for all $t\ge t_0$. Hence we may assume that 
\[
{\Ext_R^i(R/I^t,\, R)}_k\ \cong\ {\HH_I^i(R)}_k \quad \text{ for all $i \ge c+1$ and $t \ge t_0$;}
\]
in particular, by assumption, these modules are locally free over $A$.

Fix $\frakp\in\Spec A$ and $t\ge t_0$, and write $R_\frakp=R\otimes_A A_\frakp$ and $B_\frakp=B \otimes_A A_\frakp$. Consider the natural base change map 
\[
\CD
b^i(\frakp) \colon {\Ext_{R_\frakp}^i\left(R_\frakp/I^tR_\frakp,\, R_\frakp\right)}_k \otimes_{A_\frakp} \kappa(\frakp) @>>> {\Ext_{R(\frakp)}^i\left(R(\frakp)/I^tR(\frakp),\, R(\frakp)\right)}_k.
\endCD
\]
The property of exchange for local Ext’s, \cite[Theorem~1.9]{ALTMAN_KLEIMAN_COMPACT_PICARD}, \cite[Theorem~A.5]{LONSTED_KLEIMAN}, implies that, if~$b^i(\frakp)$ is surjective, then:

(a) $b^i(\frakp)$ is an isomorphism;

(b) the base change map below is also an isomorphism
\[
\CD
{\Ext_{R_\frakp}^i\left(R_\frakp/I^t R_\frakp,\, R_\frakp\right)}_k \otimes_{A_\frakp} B_\frakp @>>> {\Ext_{R_B}^i \left(R_B/I^t R_B,\, R_B\right)}_k \otimes_A A_\frakp;
\endCD
\]

(c) $b^{i-1}(\frakp)$ is surjective if and only if ${\Ext_{R_\frakp}^i\left(R_\frakp/I^t R_\frakp,\, R_\frakp\right)}_k$ is $A_\frakp$-free.

Since we know that ${\Ext_{R_\frakp}^i\left(R_\frakp/I^t R_\frakp,\, R_\frakp\right)}_k$ is $A_\frakp$-free for all $i \ge c+1$, by descending induction on $i$ (starting with $i = n+1$), we get the isomorphism
\[
\CD
{\Ext_{R_\frakp}^{c+1}\left(R_\frakp/I^t R_\frakp,\, R_\frakp\right)}_k \otimes_{A_\frakp} B_\frakp @>\cong>> {\Ext_{R_B}^{c+1}\left(R_B/I^t R_B,\, R_B\right)}_k \otimes_A A_\frakp.
\endCD
\]
Therefore, we obtain the isomorphisms 
\begin{align*}
{\Ext_{R_B}^{c+1}\left(R_B/I^t R_B,\, R_B\right)}_k \otimes_A A_\frakp
&\ \cong\ {\Ext_{R_\frakp}^{c+1}\left(R_\frakp/I^t R_\frakp,\, R_\frakp\right)}_k \otimes_{A_\frakp} B_\frakp \\
&\ \cong\ {\HH_I^{c+1}\left(R_\frakp\right)}_k \otimes_{A_\frakp} B_\frakp \\ 
&\ \cong\ {\Ext_{R_\frakp}^{c+1}\left(R_\frakp/I^{t+1}R_\frakp,\, R_\frakp\right)}_k \otimes_{A_\frakp} B_\frakp \\
&\ \cong\ {\Ext_{R_B}^{c+1}\left(R_B/I^{t+1}R_B,\, R_B\right)}_k \otimes_A A_\frakp
\end{align*}
for all $t \ge t_0$. Finally, as $\frakp \in \Spec A$ was arbitrary, we get the isomorphism 
\[
\CD
{\Ext_{R_B}^{c+1}(R_B/I^tR_B,\, R_B)}_k @>\cong>> {\Ext_{R_B}^{c+1}(R_B/I^{t+1}R_B,\, R_B)}_k
\endCD
\]
for all $t \ge t_0$. This concludes the prove of the proposition.
\end{proof}

The following lemma yields a flatness criterion, assuming \autoref{thm_CRS} and \autoref{prop_inject}.

\begin{lemma}[Valuative criterion for the flatness of cohomology]
\label{prop_valuative}
Let $R= A[x_0,\dots,x_n]$ be a standard graded polynomial ring over a Noetherian domain $A$, and $I\subset R$ be a homogeneous ideal. Fix $j \in \ZZ$ and~$t\ge 1$. Assume that:
\begin{enumerate}[\quad\rm(a)]
\item For each $A$-algebra $V$ that is a discrete valuation ring, and for each $i\ge 0$, the natural map 
\[
\CD
{\Ext_{R_V}^i\left(R_V/I^tR_V,\, R_V\right)}_{-j-n-1} @>>> {\HH_I^i(R_V)}_{-j-n-1}
\endCD
\]
is injective, where $R_V\colonequals R\otimes_A V$,
\item $R/I^t$ is $A$-flat,
\item $\HH_I^i(R)$ is $A$-flat for each $i\ge 0$.
\end{enumerate}
Then ${\HH_\frakM^i(R/I^t)}_j$ is $A$-flat for each $i \ge 0$.
\end{lemma}

\begin{proof}
Set $k\colonequals -j-n-1$. By \autoref{prop_fib_full_mod}, it suffices to show that 
\[
\rank_{\kappa(\frakp)} {\Ext_{R(\frakp)}^i\left(R(\frakp)/I^tR(\frakp),\, R(\frakp)\right)}_k
\ =\
\rank_Q {\Ext_R(R/I^t,\, R)}_k \otimes_A Q
\]
for all $\frakp\in\Spec A$, where $Q\colonequals\Frac A$.

Fix $\frakp\in\Spec A$. By \cite[Exercise~II.4.11]{HARTSHORNE} or \cite[Proposition~7.1.7]{EGAII}, there exists a discrete valuation ring $(V,\frakn)$ of $Q$ that dominates $A_\frakp$, i.e., $A_\frakp \subset V$ and $\frakp A_\frakp = \frakn\cap A_\frakp$. To simplify notation, set $\kappa(\frakn)\colonequals V/\frakn$ and $R_V(\frakn) \colonequals R_V \otimes_V \kappa(\frakn) \cong R_V/\frakn R_V$. Then $\HH_I^i(R_V) \cong \HH_I^i(R) \otimes_A V$ is flat over $V$, equivalently torsionfree over $V$, for each $i\ge 0$.
Therefore the injective map 
\[
{\Ext_{R_V}^i\left(R_V/I^tR_V,\, R_V\right)}_k\ \into\ {\HH_I^i(R_V)}_k
\]
from (a) implies that each ${\Ext_{R_V}^i\left(R_V/I^tR_V,\, R_V\right)}_k$ is also torsionfree over $V$, hence flat over $V$. By applying \autoref{prop_fib_full_mod} to $R_V/I^tR_V$ over $V$, we obtain 
\begin{align*}
\rank_{\kappa(\frakn)} {\Ext_{R_V(\frakn)}^i\left(R_V(\frakn)/I^tR_V(\frakn),\, R_V(\frakn)\right)}_k
& \ =\ \rank_{\kappa(\frakn)} {\Ext_{R_V}^i(R_V/I^tR_V,\, R_V)}_k \otimes_V \kappa(\frakn)\\
& \ =\ \rank_Q {\Ext_R^i(R/I^t,\, R)}_k \otimes_A Q.
\end{align*}
Considering the field extension $\kappa(\frakp) \into \kappa(\frakn)$, we have isomorphisms
\[
R_V(\frakn)\ \cong\ (R\otimes_A V)\otimes_V\kappa(\frakn)
\ \cong\ R\otimes_A\kappa(\frakn)
\ \cong\ (R\otimes_A\kappa(\frakp))\otimes_{\kappa(\frakp)}\kappa(\frakn)
\ \cong\ R(\frakp)\otimes_{\kappa(\frakp)} \kappa(\frakn).
\]
Thus, we have the isomorphism 
\[
\Ext_{R_V(\frakn)}^i\big(R_V(\frakn)/I^tR_V(\frakn),\, R_V(\frakn)\big)\ \cong\
\Ext_{R(\frakp)}^i\big(R(\frakp)/I^tR(\frakp),\, R(\frakp)\big) \otimes_{\kappa(\frakp)}\kappa(\frakn).
\]
Combining everything, we obtain 
\[
\rank_{\kappa(\frakp)} {\Ext_{R(\frakp)}^i\left(R(\frakp)/I^tR(\frakp),\, R(\frakp)\right)}_k\ =\
\rank_Q {\Ext_R(R/I^t,\, R)}_k \otimes_A Q,
\]
which completes the proof of the lemma.
\end{proof}

We are now ready to prove the first main result:

\begin{proof}[Proof of \autoref{thm_A}]
Let $I \subset R$ be a homogeneous ideal defining $X$ in $\PP_A^n$. By generic freeness, for example, \cite[\S24]{MATSUMURA}, there exists a nonzero element $a'$ in $A$ such that 
\[
\gr_I(R) \otimes_A A_{a'}\ =\ \bigoplus_{t\ge 0}\, I^t/I^{t+1} \otimes_A A_{a'}
\] 
is $A_{a'}$-free. By \autoref{thm_CRS}, there exists a nonzero element $a''$ in $A$ such that $\HH_I^i(R) \otimes_A A_{a''}$ is $A_{a''}$-free for each $i \ge 0$. To simplify notation, substitute $A$ by the localization $A_{a'a''}$. Set $k = -j-n-1$. Then \autoref{prop_inject} yields an integer $t_0 > 0$ such that the natural map
\[
\CD
{\Ext_{R_V}^i\left(R_V/I^tR_V,\, R_V\right)}_k @>>> {\HH_I^i\left(R_V\right)}_k
\endCD
\]
is injective for all $t \ge t_0$ and $i \ge 0$ and for any $A$-algebra $V$ that is a discrete valuation ring. From \autoref{prop_valuative}, we get that ${\HH_\frakM^i\left(R/I^t\right)}_j$ is $A$-flat for all $i \ge 0$ and $t \ge t_0$. After possibly increasing $t_0$, we may also assume that $\left[I^t/I^{t+1}\right]_j=0$ for all $t\ge t_0$. This implies that 
\[
\HH^i\big(X_t,\, \calO_{X_t} \otimes_{\calO_{\PP_A^n}}\calF\big)
\]
is $A$-flat for all $i \ge 0$ and $t \ge t_0$, cf.~proof of \autoref{cor_vanish_ext}. Finally, by generic freeness, there exists a nonzero element $a \in A$ such that
\[
\HH^i\big(X_t,\, \calO_{X_t} \otimes_{\calO_{\PP_A^n}}\calF\big) \otimes_A A_a
\]
is $A_a$-free for all $1 \le t < t_0$ and $i \ge 0$. This concludes the proof of the theorem.
\end{proof}

\section{Ideals of points}
\label{sect_points}

Towards investigating \autoref{question_constant} via the cohomology of thickenings, consider the following setup. Let $\ba\colonequals\left\{a_{ij} \mid 1\le i\le m,\ 0\le j \le n\right\}$ be indeterminates over a field $\kk$. Set $A\colonequals\kk[\ba]$, and let $R$ be the polynomial ring $A[x_0,\dots,x_n]$. Consider the ideal $I \colonequals \frakA_1 \cap \dots \cap \frakA_m$ of $R$, where
\[
\frakA_i \colonequals I_2 \begin{pmatrix}
x_0 & x_1 & \cdots & x_n\\
a_{i0} & a_{i1} & \cdots & a_{in}
\end{pmatrix}
\quad \text{ for } 1\le i\le m,
\] 
i.e., $\frakA_i$ is the ideal generated by size $2$ minors, and $I$ is the defining ideal of a \emph{generic} set of $m$ points in projective space. Set $\Rees(I)$ to be the Rees algebra $\bigoplus_{t\ge 0} I^tT^t \subset R[T]$. If there exists a nonzero element $a \in A$ such that 
\[
\HH^i_{(x_0,\dots,x_n)}\left(\Rees(I)\right)\otimes_A A_a
\]
is a free $A_a$-module for each $i\ge 0$, then the answer to \autoref{question_constant} is affirmative. In the case that $m\le n+2$, the answer to \autoref{question_constant} is readily seen to be affirmative, since $\PGL_{n+1}(\kk)$ acts transitively on $(n+2)$-tuples of points in $\PP_\kk^n$ that are in general linear position. We note that the stronger statement, i.e., the generic freeness of the relevant local cohomology modules, also holds in this case:

\begin{theorem}
\label{thm_points}
Let $\kk$ be an arbitrary field. Fix integers $m\le n+2$, and consider the rings $A$, $R$, and $\Rees(I)$, as defined earlier in this section. Then there exists a nonzero element $a \in A$ such that 
\[
\HH^i_{(x_0,\dots,x_n)}\left(\Rees(I)\right)\otimes_A A_a
\]
is a free $A_a$-module for each $i\ge 0$.
\end{theorem}

\begin{proof}
The ideals $\frakA_1,\dots\frakA_m$ are unchanged by column operations on the matrix
\[
\begin{pmatrix}
x_0 & x_1 & \cdots & x_n\\ 
a_{10} & a_{11} & \cdots & a_{1n}\\
a_{20} & a_{21} & \cdots & a_{2n}\\
\vdots & \vdots & & \vdots \\
a_{m0} & a_{m1} & \cdots & a_{mn}
\end{pmatrix}.
\]
Note that $\frakA_i$ is the ideal of size $2$ minors of the submatrix given by the first and $(i+1)$-th rows. After inverting the element $a_{10}\in A$, we may subtract multiples of the first column from the other columns, so that the other entries in that row are zero. Next invert the element diagonally below $a_{10}$, and subtract suitable multiples of the second column from the others, so as to obtain zeros elsewhere in that row; proceed in this manner. Let $a$ denote the product of the elements that have been inverted. For notational simplicity, assume for the moment that $m=n+2$; the resulting matrix then has the form 
\[
\begin{pmatrix}
y_0 & y_1 & y_2 &\cdots & y_n\\ 
a_1 & 0 & 0 & \cdots & 0\\
0 & a_2 & 0 & \cdots & 0\\
0 & 0 & a_3 & & 0\\
\vdots & \vdots & & \ddots \\
0 & 0 & 0 & & a_{n+1}\\
b_0 & b_1 & b_2 &\cdots & b_n
\end{pmatrix},
\]
where $a_i,b_i\in A_a$, and $y_0,\dots,y_n$ are independent linear forms in the indeterminates $x_0,\dots,x_n$. Specifically, the ideal $(x_0,\dots,x_n)R_a$ agrees with the ideal $(y_0,\dots,y_n)R_a$. After replacing the element $a$, we may also assume that the $b_i$ are units in $A_a$; dividing the columns by $b_0,\dots,b_n$ respectively, and renaming the elements $a_i$ and $y_i$, the matrix takes the form
\[
\begin{pmatrix}
y_0 & y_1 & y_2 &\cdots & y_n\\ 
a_1 & 0 & 0 & \cdots & 0\\
0 & a_2 & 0 & \cdots & 0\\
0 & 0 & a_3 & & 0\\
\vdots & \vdots & & \ddots \\
0 & 0 & 0 & & a_{n+1}\\
1 & 1 & 1 & \cdots & 1
\end{pmatrix}.
\]
Note that
\[
\frakA_iR_a\ =\ (y_0,\ \dots,\ \widehat{y}_{i-1},\ \dots,\ y_n)R_a\quad\text{ for }1\le i\le n+1,
\]
and that
\[
\frakA_{n+2}R_a\ =\ (y_1-y_0,\ y_2-y_0,\ \dots,\ y_n-y_0)R_a.
\]
In the case that $m<n+2$, the generators of the ideals $\frakA_1R_a,\dots,\frakA_mR_a$ remain as displayed above, while the other ideals $\frakA_iR_a$ do not occur. In either case, $\frakA_1R_a,\dots,\frakA_mR_a$ are extended from the polynomial ring $P \colonequals \kk[y_0,\dots,y_n]$; set $J$ to be the intersection of the corresponding ideals in $P$, in which case
\[
IR_a\ =\ \frakA_1R_a\cap\dots\cap\frakA_mR_a\ =\ JR_a.
\]
Set $\Rees(J)$ to be the Rees algebra $\bigoplus_{t \ge 0} J^tT^t \subset P[T]$. For each $i\ge 0$, the local cohomology module
\[
\HH^i_{(y_0,\dots,y_n)}\left(\Rees(J)\right)
\]
is, of course, free over the field $\kk$, but then
\[
\HH^i_{(y_0,\dots,y_n)}\left(\Rees(J)\right) \otimes_\kk A_a\ \cong\ \HH^i_{(x_0,\dots,x_n)}\left(\Rees(I)\right) \otimes_A A_a
\]
is a free $A_a$-module.
\end{proof}

In the next section, we prove that in the case of nine points, $\HH^2_{(x_0,x_1,x_2)}\left(\Rees(I)\right)$ is not generically free over the coefficient ring $A$.

\section{The erratic behavior of nine points in $\PP^2$}
\label{nine_points}

We show that \autoref{thm_points} cannot be extended to the relatively simple case of nine points in projective plane. Throughout this section, we work over the field $\CC$ of complex numbers. Consider the set of indeterminates $\ba\colonequals\left\{a_{ij} \mid 1\le i\le 9,\ 0\le j \le 2\right\}$ over $\CC$, and set $A\colonequals\CC[\ba]$. Let $R$ be the polynomial ring $A[x_0,x_1,x_2]$, and for $1\le i\le 9$, consider the ideals
\[
\frakA_i \colonequals I_2 \begin{pmatrix}
x_0 & x_1 & x_2 \\
a_{i0} & a_{i1} & a_{i2}
\end{pmatrix}.
\]
Set $I\colonequals\frakA_1\cap\dots\cap\frakA_9$, and let $\scrI\subset\calO_{\PP_A^2}$ be the corresponding ideal sheaf.
Let $\Rees(I) = \bigoplus_{t\ge 0} I^tT^t \subset R[T]$ and $\Rees(\scrI) = \bigoplus_{t\ge 0} \scrI^tT^t \subset \calO_{\PP_A^2}[T]$ be the respective Rees algebras. For integers $t \ge 1$, set
\[
M_t \colonequals \HH^1\big(\PP_A^2,\, \scrI^{t}(3t)\big).
\]
Our main result in this section is the following erratic behavior of nine point in projective plane:

\begin{theorem}
\label{thm_9points}
With the notation as above, the following hold: 
\begin{enumerate}[\quad\rm(a)]
\item There is no nonzero element $a \in A$ such that $\bigoplus_{t \ge 1} M_t \otimes_A A_a$ is a flat $A_a$-module.
\item The set of associated prime ideals
\[
\bigcup_{t \ge 1} \Ass_A\left(M_t\right)
\]
is an infinite set. 
As a consequence, each of the following sets of associated primes is infinite:
\[
\Ass_A\left(\HH_{(x_0,x_1,x_2)}^2\left(\Rees(I)\right)\right), \quad
\Ass_R\left(\HH_{(x_0,x_1,x_2)}^2\left(\Rees(I)\right)\right),
\text{\quad and \quad}
\Ass_{\Rees(I)}\left(\HH_{(x_0,x_1,x_2)}^2\left(\Rees(I)\right)\right).
\]
\end{enumerate}
\end{theorem}

\begin{proof}
The proof occupies the rest of the section; by way of contradiction, we assume the following hypothesis:

\begin{hypothesis}
\label{hyp}
Assume there exists a nonzero element $a \in A$ such that $\calP \colonequals \bigcup_{t \ge 1} \Ass_{A_a}\left(M_t \otimes_A A_a\right)$ is a finite set. 
\end{hypothesis}

The proof of \autoref{thm_9points} is divided into several steps. 

\smallskip
\noindent
{\sc Step 1:} \emph{Consider the support of each $M_t$.} For $t\ge 1$, we analyze the support 
\[
\Supp_{A_a}\left(M_t \otimes_A A_a\right)\ =\ \big\lbrace \frakp \in \Spec(A_a) \mid \left(M_t\right)_\frakp \neq 0 \big\rbrace\ =\ V\left(\Ass_{A_a}\left(M_t\otimes_A A_a\right)\right).
\]
Since $\calP$ is assumed to be a finite set, \autoref{hyp}, there exists an integer $t_0 \ge 1$ such that 
\begin{equation}
\label{eq_contain_supp}
\Supp_{A_a}\left(M_t \otimes_A A_a\right)\ \subseteq\ \Supp_{A_a}\left(M_1 \otimes_A A_a\right) \,\cup\, \cdots \,\cup\, \Supp_{A_a}\left(M_{t_0} \otimes_A A_a\right)
\end{equation}
for each $t \ge 1$.

\smallskip
\noindent
{\sc Step 2:} \emph{Find a dense open subset where we can specialize.} We regard a closed point $S \in \AA_\CC^{27} = \left(\AA_\CC^3\right)^9$ as a tuple of $9$ points in $\AA_\CC^3$. For each $S \in \AA_\CC^{27}$, consider the fiber $\PP_{\kappa(S)}^2 = \PP_A^2 \times_A \kappa(S)$ and the corresponding ideal sheaf $\scrI_S \colonequals \scrI \cdot \calO_{\PP_{\kappa(S)}^2}$, where $\frakp_S \colonequals I(S) \subset \maxSpec(A)$ and $\kappa(S) \colonequals A_{\frakp_S}/\frakp_SA_{\frakp_S}$ is the residue field. Since we consider closed points, we always have $\kappa(S) \cong \CC$. By generic freeness, we may choose a dense open subset $U \subset \AA_\CC^{27}$ such that the following conditions hold:

\begin{enumerate}[\rm (a)]
\item The quotient $\calO_{\PP_{A}^2}[T] / \Rees(\scrI)$ is flat over $U$; this may be achieved by the Hochster-Roberts generic freeness result, \cite[Lemma~8.1]{HOCHSTER_ROBERTS1}. It then follows that $\Rees(\scrI)= \bigoplus_{t \ge 1}\scrI^tT^t$ is flat over $U$, and that
\[
\scrI_S^t\ =\ \scrI^t \cdot \calO_{\PP_{\kappa(S)}^2}\ \cong\ \scrI^t \otimes_A \kappa(S)
\]
for all $S \in U$ and $t \ge 1$.

\item For each $S \in U$, we have $\scrI_S = \scrI \cdot \calO_{\PP_{\kappa(S)}^2} = \big( \widetilde{\frakA_1} \cdot \calO_{\PP_{\kappa(S)}^2}\big) \cap \cdots \cap \big( \widetilde{\frakA_9} \cdot \calO_{\PP_{\kappa(S)}^2}\big)$; this may be achieved by enforcing that all the sheaves in the exact sequence 
\[
\CD
0 @>>> \scrI @>>> \calO_{\PP_{A}^2} @>\psi>> \bigoplus_{i=1}^9 \calO_{\PP_{A}^2}/\widetilde{\frakA_i} @>>> {\rm Coker}(\psi) @>>> 0
\endCD
\]
are flat over $U$.

\item Consider the natural smooth morphism $\pi\colon \left(\AA_\CC^3\setminus \{0\}\right)^9 \to \left(\PP_{\CC}^2\right)^9$. We may assume $U \subset \left(\AA_\CC^3\setminus \{0\}\right)^9$, so $\pi(U)$ is a dense open subset of $\left(\PP_{\CC}^2\right)^9$.

\item By an abuse of notation, and by possibly restricting $U \subset \AA_\CC^{27}$ further, we may assume that each element $S \in \pi(U)$ represents a set of $9$ \emph{distinct} points in $\PP_\CC^2$. Thus, for all $t \ge 1$, $tS$ denotes the closed subscheme of $\PP_\CC^2$ determined by the ideal sheaf $\scrI_S^t \subset \calO_{\PP_\CC^2}$.

\item For each $S \in \pi(U)$, we may assume that there exists a unique cubic $C \subset \PP_{\CC}^2$ passing through the nine points of $S$, and that $C$ is smooth. For further related results, see \cite{NAGATA_CONJ, HARBOURNE}.

\item We have $U \subset D(a)$, where $a \in A$ is the nonzero element guaranteed by \autoref{hyp}.
\end{enumerate}
For the remainder of this proof, we fix a dense open subset $U \subset \AA_\CC^{27}$ that satisfies the above conditions. 

Next, note that using the exact sequence
\[
\CD
0 @>>> \scrI_S^t(3t) @>>> \calO_{\PP_\CC^2}(3t) @>>> \calO_{tS}(3t) @>>> 0,
\endCD
\]
we obtain
\begin{equation}
\label{eq_vanish_H2}
\HH^i\left(\PP_\CC^2,\, \scrI_S^t(3t)\right)\ =\ 0.
\end{equation}
for all $i \ge 2$, $S \in U$ and $t \ge 1$. 

By base change of cohomology, for example, \cite[Theorem~12.11]{HARTSHORNE}, we have
\begin{equation}
\label{eq_specialization}
M_t \otimes_A \kappa(S)\ =\ \HH^1\left(\PP_A^2,\, \scrI^{t}(3t)\right) \otimes_A \kappa(S)\
\cong\ \HH^1\left(\PP_{\CC}^2,\, \scrI_S^t(3t)\right)
\end{equation}
for all $S \in U$ and $t \ge 1$, where we use the vanishing in \autoref{eq_vanish_H2}. Using \autoref{eq_specialization} and Nakayama's lemma, it follows that
\[
\Supp_A(M_t) \cap U\ =\ \Big\lbrace S \in U \,\mid\, \HH^1\left(\PP_{\CC}^2,\, \scrI_S^t(3t)\right) \neq 0 \Big\rbrace.
\]
Consequently, \autoref{eq_contain_supp} yields the following useful lemma:

\begin{lemma}
\label{lem_vanish_t0}
Assume \autoref{hyp} and let $S \in U$. If $\HH^1\big(\PP_{\CC}^2,\, \scrI_S^k(3k)\big) = 0$ for all $1 \le k \le t_0$, then
\[
\HH^1\big(\PP_{\CC}^2,\, \scrI_S^t(3t)\big) = 0 \quad \text{ for all $t \ge 1$.}
\]
\end{lemma}

\smallskip
\noindent
{\sc Step 3:} \emph{Restrict to a smooth cubic.} Choose a smooth cubic $C \subset \PP_\CC^2$ that passes through a set $S$ of nine points in $\pi(U) \subset \left(\PP_{\CC}^2\right)^9$. Note that~$C$ is in fact an elliptic curve, i.e., that it has genus one. Consider the dense open subset
\[
V \colonequals C^9 \cap \pi(U)\ \subset\ C^9\ =\ C \times \cdots \times C,
\]
and the summation map $\Phi\colon V \to \Pic^0(C) \cong C$ given by 
\[
\Phi(S) \colonequals 3\left[H\right] - \sum_{i=1}^9 \left[P_i\right],
\]
where $S = (P_1,\dots,P_9) \in V$, and $[H]$ is the hyperplane class. Note that the map $\Phi$ is dominant. 

Given $S = (P_1,\dots,P_9)\in V$, consider the divisor $D_S \colonequals 3H - \sum_{i=1}^9P_i$ on $C$, and the corresponding line bundle $\scrL_S \colonequals \calO_C(D_S)$. The residual scheme $\Res_C(tS) \subset \PP_\CC^2$ of $tS = V(\scrI_S^t)$ with respect to $C$ is the closed subscheme determined by the ideal sheaf $\scrI_S^{t}:\scrI_C$ in $\calO_{\PP_\CC^2}$. We have the exact sequence 
\[
\CD
0 @>>> \scrI_{\Res_C(tS)} \otimes \calO_{\PP_\CC^2}(-C) @>>> \scrI_S^t @>>> \scrI_{(tS) \cap C, C} @>>> 0.
\endCD
\]
Since $S \subset \PP_\CC^2$ is smooth, the normal cone $\sSpec_S\big(\bigoplus_{t \ge 1}\scrI_S^t/\scrI_S^{t+1}\big)$ is a vector bundle, and we get the equality $\scrI_{\Res_C(tS)} = \scrI_S^{t-1}$. By construction, we have $\scrI_{(tS) \cap C, C} \cong \calO_C\left(-\sum_{i=1}^9 P_i\right)$. It follows that for each $t\ge 1$, there is an exact sequence
\begin{equation}
\label{eq_residual_ses}
\CD
0 @>>> \scrI_S^{t-1}(3t-3) @>>> \scrI_S^{t}(3t) @>>> \calO_C(tD_S) @>>> 0.
\endCD
\end{equation}

To simplify notation in what follows, consider the function
\[
\gamma(t, S) \colonequals \begin{cases}
1 & \text{ if }\ \scrL_S^{\otimes t} = \calO_C(tD_S) \cong \calO_C,\\
0 & \text{ otherwise},
\end{cases}
\]
that records whether $\scrL_S = \calO_C(D_S)$ is a $t$-torsion point in $\Pic^0(C)$. By the Riemann-Roch theorem, we obtain that
\begin{equation}
\label{eq_riemann_roch}
h^0\left(C,\, \calO_C(tD_S)\right)\ =\ h^1\left(C,\, \calO_C(tD_S)\right)\ =\ \gamma(t, S),
\end{equation}
see \cite[Lemma~IV.1.2 and Theorem~IV.1.3]{HARTSHORNE}.

The following proposition serves as the main technical tool:

\begin{proposition}
Let $S \in V$.
For each $t \ge 1$, we have
\[
h^0\left(\PP_\CC^2,\, \scrI_S^{t}(3t)\right)\ =\ 1 + \sum_{k=1}^t \gamma(k, S)
\qquad \text{ and } \qquad
h^1\left(\PP_\CC^2,\, \scrI_S^{t}(3t)\right)\ =\ \sum_{k=1}^t \gamma(k, S).
\]
\end{proposition}

\begin{proof}
Let $t \ge 1$. From \autoref{eq_residual_ses} and \autoref{eq_vanish_H2}, we get the long exact sequence in cohomology
\begin{equation}
\label{eq_les_cubic}
\CD
0 @>>> \HH^0\left(\PP_\CC^2,\, \scrI_S^{t-1}(3t-3)\right) @>>> \HH^0\left(\PP_\CC^2,\, \scrI_S^{t}(3t)\right) @>\rho_t>> \HH^0\left(C,\, \calO_C(tD_S)\right) @>>>\\
@>>> \HH^1\left(\PP_\CC^2,\, \scrI_S^{t-1}(3t-3)\right) @>>> \HH^1\left(\PP_\CC^2,\, \scrI_S^{t}(3t)\right) @>>> \HH^1\left(C,\, \calO_C(tD_S)\right) @>>> 0.
\endCD
\end{equation}
Using \autoref{eq_riemann_roch}, and proceeding inductively on $t \ge 1$, it suffices to show that the map $\rho_t$ above is surjective. Indeed, for the initial step, we have $h^0(\PP_{\CC}^2,\, \calO_{\PP_{\CC}^2}) = 1$ and $h^1(\PP_{\CC}^2,\, \calO_{\PP_{\CC}^2}) = 0$.

If $\calO_C(tD_S) \not\cong \calO_C$, then $h^0\left(C,\, \calO_C(tD_S)\right) = 0$ and it follows that $\rho_t$ is surjective. Therefore, we may assume that $\calO_C(tD_S) \cong \calO_C$. Fix the assumed isomorphism $\calO_C(tD_S) \cong \calO_C$, and let $\phi_t \in \HH^0\left(C,\, \calO_C(tD_S)\right)$ be the global section corresponding with $1 \in \HH^0\left(C,\, \calO_C\right)$.

Let $d$ be the precise order of torsion of $\scrL_S = \calO_C(D_S)$, i.e., $\scrL_S$ is $d$-torsion but it is not $k$-torsion for any $k \le d -1$. This implies that $t = d \cdot a$ for some integer $a \ge 1$. Since $\HH^1\left(C, \calO_C(kD_S)\right) = 0$ for all $k \le d-1$, see \autoref{eq_riemann_roch}, proceeding inductively with the exact sequence \autoref{eq_les_cubic}, we obtain the vanishing
\[
\HH^1\left(\PP_\CC^2,\, \scrI_S^{d-1}(3d-3)\right)\ =\ 0.
\]
But then \autoref{eq_les_cubic} yields that $\rho_d$ is surjective. Consequently, we may choose $F \in \HH^0\left(\PP_\CC^2,\, \scrI_S^{d}(3d)\right)$ such that $\rho_d(F) = \phi_d$. 
Let 
\[
\CD
\alpha\colon \HH^0\left(\PP_\CC^2,\, \scrI_S^{d}(3d)\right)^{\otimes a} @>>> \HH^0\left(\PP_\CC^2,\, \scrI_S^{t}(3t)\right)
\endCD
\]
and 
\[
\CD
\beta\colon \HH^0\left(C,\, \calO_C(dD_S)\right)^{\otimes a} @>>> \HH^0\left(C,\, \calO_C(tD_S)\right)
\endCD
\]
be the natural $a$-fold multiplication maps. We then have the commutative diagram
\[
\minCDarrowwidth45pt
\CD
\HH^0\left(\PP_\CC^2,\, \scrI_S^{d}(3d)\right)^{\otimes a} @>\rho_d^{\otimes a}>> \HH^0\left(C,\, \calO_C(dD_S)\right)^{\otimes a} @. \quad \cong\ \HH^0\left(C,\, \calO_C\right)^{\otimes a} \\
@V{\alpha}VV @VV{\beta}V @.\\
\HH^0\left(\PP_\CC^2,\, \scrI_S^{t}(3t)\right) @>\rho_t>> \HH^0\left(C,\, \calO_C(tD_S)\right) @. \cong\ \HH^0\left(C,\, \calO_C\right).
\endCD
\]
We conclude that $\phi_t = \rho_t\left(F^a\right)$, so $\rho_t$ is surjective, as required.
\end{proof}

As a consequence of the above proposition we obtain the following vanishing criterion:

\begin{corollary}
\label{prop_residual_technique}
Let $S \in V$ and fix $t \ge 1$. Then $\HH^1\big(\PP_\CC^2,\, \scrI_S^t(3t)\big) \neq 0$ if and only if $k \cdot \left[\scrL_S\right] = 0$ in $\Pic^0(C)$, for some integer $k$ with $1 \le k \le t$.
\end{corollary}

For each $k \ge 1$, consider the following locus
\[
Z_k \colonequals \big\lbrace S \in V \mid k \cdot \left[\scrL_S\right] = 0 \text{ in }\Pic^0(C) \big\rbrace.
\]
Combining \autoref{lem_vanish_t0} and \autoref{prop_residual_technique}, it follows that 
\begin{equation}
\label{eq_finite_union_Z}
\bigcup_{k = 1}^\infty Z_k\ =\ \bigcup_{k = 1}^{t_0} Z_k. 
\end{equation}
In terms of the map $\Phi\colon V \to \Pic^0(C)$, this means that if a configuration $S \in V$ is mapped to a torsion point, it must be a torsion point of order at most $t_0$.

\smallskip
\noindent
{\sc Step 4:} \emph{Exploit the group structure of $\Pic^0(C)$.} Since the map $\Phi\colon V \to \Pic^0(C)$ is dominant, and $\Pic^0(C) \cong C$ is a smooth curve, $\Pic^0(C) \setminus \Phi(V)$ must be a finite set. The torsion subgroup of $\Pic^0(C)$ is
\[
\Pic^0(C)_{\rm tor}\ \cong\ \lim\limits_{\substack{\to\\ n}} \left(\ZZ/n\ZZ\right)^2\ \cong\ \left(\QQ/\ZZ\right)^2,
\]
see, for example, \cite[Chapter~2]{MUMFORD_ABELIAN}. Therefore, for each $k \ge 1$, the set of torsion points
\[
Y_k \colonequals \big\lbrace w \in \Pic^0(C) \mid k \cdot w = 0 \text{ and } j \cdot w \neq 0 \text{ for all } 1 \le j \le k-1 \big\rbrace
\]
of order $k$ is nonempty. This implies that 
\[
\Phi(V) \,\cap\, \left( \bigcup_{k = t_0+1}^\infty Y_k \right)\ \neq\ \varnothing.
\]
This contradicts \autoref{eq_finite_union_Z}, completing the proof of \autoref{thm_9points}.
\end{proof}

With similar notation as before, for any $S = (P_1,\ldots, P_9) \subset \left(\PP_\CC^2\right)^9$, set $\scrI_S \subset \calO_{\PP_\CC^2}$ to be the ideal sheaf of the set $\lbrace P_1,\ldots,P_9\rbrace \subset \PP_\CC^2$, and $I_S \colonequals \bigoplus_{k\ge 0} \HH^0\left(\PP_\CC^2,\, \scrI_S(k)\right)$ be the corresponding saturated homogeneous ideal. For each $t \ge 1$, the $t$-th symbolic $I_S^{(t)}$ of $I_S$ coincides with the ideal $\bigoplus_{k\ge 0} \HH^0\left(\PP_\CC^2,\, \scrI_S^t(k)\right)$. Due to classic results of Nagata (see \cite{NAGATA_CONJ,HARBOURNE}), we know that, if $S$ is a general set of nine points in $\PP_\CC^2$, then 
\[
\rank_\CC\left(\left[I_S^{(t)}\right]_{3t}\right)\ =\ h^0\left(\PP_\CC^2,\, \scrI_S^t(3t)\right)\ \ge\ 1 \quad \text{ for all } \quad t \ge 1.
\]
However, \autoref{thm_9points} yields the following conclusion: one cannot expect constant behavior for all $t \ge 1$.

\begin{corollary}
\label{cor_symb_pow}
There is no dense open subset $U \subset \left(\PP_\CC^2\right)^9$ such that, for all $t \ge 1$, the function 
\[
S \in \left(\PP_\CC^2\right)^9\ \mapsto\ h^0\left(\PP_\CC^2,\, \scrI_S^t(3t)\right)
\] 
is constant on $U$.
\end{corollary}

\begin{proof}
The proof follows by combining \autoref{thm_9points}, the local constancy of the Euler characteristic of the fibers of a projective flat family \cite[\S 5]{MUMFORD_ABELIAN}, and Grauert's theorem \cite[Corollary III.12.9]{HARTSHORNE}.
\end{proof}

\section*{Acknowledgments}

We are grateful to Mircea Musta\c t\u a for comments on an earlier version of this manuscript.


\end{document}